\documentclass[12pt]{amsart}

\usepackage[usenames,dvipsnames]{xcolor}

\usepackage{enumitem,kantlipsum}
\usepackage{hyperref}
\usepackage{mathtools}
\usepackage{xfrac}
\usepackage[utf8]{inputenc}
\usepackage{amssymb}
\usepackage{amsmath, amsthm}
\usepackage{amsfonts}
\usepackage{comment}
\usepackage{tikz}
\usepackage{algorithmic}
\usepackage{asymptote}
\usepackage{url}
\usepackage{mathrsfs}

\newtheorem{theorem}{Theorem}[section]
\newtheorem{corollary}[theorem]{Corollary}
\newtheorem{lemma}[theorem]{Lemma}
\newtheorem{proposition}[theorem]{Proposition}
\newtheorem{conjecture}[theorem]{Conjecture}
\newtheorem{definition}[theorem]{Definition}

\usepackage[OT2,T1]{fontenc}
\DeclareSymbolFont{cyrletters}{OT2}{wncyr}{m}{n}
\DeclareMathSymbol{\Sha}{\mathalpha}{cyrletters}{"58}

\theoremstyle{definition}

\newtheorem{example}[theorem]{Example}
\newtheorem{remark}[theorem]{Remark}

\newcommand{\Pt}{\mathbb{P}^2}
\newcommand{\Pm}{\mathbb{P}^m}

\newcommand{\Ann}{{\rm Ann}}

\newcommand{\pmfq}{\PPP^m(\fq)}

\newcommand{\N}{\mathbb{N}}

\newcommand{\PPP}{\mathbb{P}}
\newcommand{\F}{\mathbb{F}}
\newcommand{\fq}{\F_q}

\newcommand{\OO}{\mathcal{O}}

\newcommand{\Spec}{\text{Spec}}

\newcommand{\m}{\mathfrak{m}}

\newcommand{\mult}{\operatorname{mult}}
\newcommand{\Supp}{\operatorname{Supp}}

\newcommand{\mm}{\mathtt{m}}

\newcommand\restr[2]{{
  \left.\kern-\nulldelimiterspace 
  #1 
  \vphantom{\big|} 
  \right|_{#2} 
  }}

\newtheorem{innercustomthm}{Theorem}
\newenvironment{customthm}[1]
  {\renewcommand\theinnercustomthm{#1}\innercustomthm}
  {\endinnercustomthm}

\newtheorem{innercustomprop}{Proposition}
\newenvironment{customprop}[1]
  {\renewcommand\theinnercustomprop{#1}\innercustomprop}
  {\endinnercustomprop}

\begin{document}
\title{Largest zero-dimensional intersection of $r$ degree $d$ hypersurfaces}
\author{Yuxin Lin, Deepesh Singhal}

\begin{abstract}
Suppose we have $r$ hypersurfaces in $\mathbb{P}^m$ of degree $d$, whose defining polynomials are linearly independent and their intersection is zero-dimensional. Then what is the maximum number of points in the intersection of the $r$ hypersurfaces?
We conjecture an exact formula for this problem and prove it when $m=2$.

We show that this can be used to compute the generalized Hamming weights of the projective Reed-Muller code $\operatorname{PRM}_q(d,2)$ and hence settle a conjecture of Beelen, Datta, and Ghorpade for $m=2$.
\end{abstract}

\maketitle

\section{Introduction}

Let $\fq$ be the finite field of size $q$.
Let $S(m,\fq)=\fq[x_0,\dots,x_m]$ and $S_d(m,\fq)$ be its $d^{th}$ graded component. Given a positive integer $r \leq \binom{m+d}{d}=\dim_{\fq}(S_d(m,\fq))$, Beelen, Datta and Ghorpade in \cite{beelen2018maximum} define
\[
e_r(d,m;q)
:=\max\{|V(W)(\fq)| : W\subseteq S_d(m,\fq), \dim(W)=r\}.
\]
We would like to note that the problem of computing $e_r(d,m;q)$ is equivalent to computing the $r^{th}$ generalized Hamming weight of the projective Reed-Muller code $PRM_q(d,m)$.
Beelen, Datta and Ghorpade have conjectured an exact formula for $e_r(d,m;q)$. We will introduce some notation to state their conjecture.

Let $\N$ be the set of non-negative integers. Beelen, Datta and Ghorpade in \cite{beelen2018maximum} define
\[
\Omega(d,m)
:=\Big\{
(\gamma_1,\dots,\gamma_{m+1})\in \N^{m+1}
: \sum_{i=1}^{m+1} \gamma_i=d
\Big\}.
\]
For $1\leq r\leq |\Omega(d,m)|=\binom{m+d}{d}$, let $\omega_r(d,m)=(\beta_1,\dots,\beta_{m+1})$ be its $r^{th}$ largest element under lexicographical ordering. Then, they define
\[
H_r(d,m;q)
:=\sum_{i=1}^m \beta_iq^{m-i}.
\]
Let $\pi_m(q):=|\pmfq|=\frac{q^{m+1}-1}{q-1}$. If $m<0$, we set $\pi_m(q)=0$.

\begin{conjecture}[Beelen-Datta-Ghorpade Conjecture]\label{Conj: complete GDC}\cite{beelen2022combinatorial}
Suppose $m,d,r$ are positive integers that satisfy $1 \leq r \leq \binom{m+d}{d}$. Pick the unique $1\leq l\leq m+1$, for which
$$\tbinom{m+d}{d}-\tbinom{m+d+1-l}{d} <r \leq \tbinom{m+d}{d}-\tbinom{m+d-l}{d}.$$
Let $j=r-\binom{m+d}{d}+\binom{m+d+1-l}{d}$, so $0< j\leq \binom{m+d-l}{d-1}$.
Then for $q\geq d+1$ we have
\[
e_r(d,m;q)
=H_j(d-1,m-l+1;q) +\pi_{m-l}(q).
\]
\end{conjecture}
We denote this conjectured formula as
\[
f_r(d,m;q)
:=H_j(d-1,m-l+1;q) +\pi_{m-l}(q).
\]
For fixed $m$, $d$ and $r$, it is a polynomial in $q$.
We show that there is a simpler way of describing this polynomial.
\begin{proposition}\label{Prop: simple expression frdm}
Given positive integers $m,d,r$ that satisfy $1\leq r\leq \binom{m+d}{d}$. Suppose $\omega_r(d,m)=(\beta_1,\dots,\beta_{m+1})$. Let $l$ be the smallest index for which $\beta_l\neq 0$. Then we have
$$\tbinom{m+d}{d}-\tbinom{m+d+1-l}{d} <r \leq \tbinom{m+d}{d}-\tbinom{m+d-l}{d},$$
and
\[f_r(d,m;q)=H_r(d,m;q)+\pi_{m-l-1}(q).\]
\end{proposition}

We would like to note that certain special cases of Conjecture~\ref{Conj: complete GDC} have been proven.
\begin{enumerate}
    \item The case $r=1$ was shown by Serre in \cite{serre1991lettre} and S{\o}rensen in \cite{sorensen1991projective}.
    \item The case $r=2$ was proven by Boguslavsky in \cite{boguslavsky1997number}.
    \item The case $d=2$ was proven by Zanella in \cite{zanella1998linear}.
    \item The case $r\leq m+1$, the case $d=1$ and the case $m=1$ were proven by Datta and Ghorpade in \cite{datta2017number}.
    \item The case $r\leq \binom{m+2}{2}$ was proven by Beelen, Datta and Ghorpade in \cite{beelen2018maximum}.
    \item The case $\binom{m+d}{d}-d\leq r\leq \binom{m+d}{d}$ was proven by Datta and Ghorpade in \cite{datta2017remarks}. This range of $r$ corresponds to $l\in \{m,m+1\}$.
\end{enumerate}
All of these cases involved specific ranges of $r$ or special values of $m,d$, but allowed for any $q\geq d+1$.
In \cite{singhal2025conjecturebeelendattaghorpade}, the authors showed that Conjecture~\ref{Conj: complete GDC} is true for any $m$, $d$ and $r$ when $q$ is sufficiently large.

\begin{theorem}\cite[Theorem 1.11]{singhal2025conjecturebeelendattaghorpade}
Suppose we are given positive integers $m,d,r$ such that $1 \leq r \leq \binom{m+d}{d}$. Pick the unique $1\leq l\leq m+1$ and $1\leq c\leq d$ for which
\[\tbinom{m+d}{d}-\tbinom{m+d+1-l}{d}+\tbinom{m+d-l-c}{d-c-1} <r 
\leq \tbinom{m+d}{d}-\tbinom{m+d+1-l}{d}+\tbinom{m+d+1-l-c}{d-c}.\]
If
$$q\geq \max\Big\{2(m-l+1)c^2+1,\;8\frac{d^{l+1}}{c},\; 164 c^{14/3}\Big\},$$
then we have
$$e_r(d,m;q)= f_r(d,m;q).$$
\end{theorem}

\begin{theorem}\cite[Theorem 1.12]{singhal2025conjecturebeelendattaghorpade}\label{Thm l=1 with e}
Suppose $m\geq 2$, $d\geq 2$, $0\leq e\leq d-2$ and $\binom{m+e}{e}< r \leq \binom{m+e+1}{e+1}$.
If $q\geq \max\{d+e+\frac{e^2-1}{d-(e+1)}, d-1+e^2-e\}$, then we have
$$e_r(d,m;q)= f_r(d,m;q).$$
\end{theorem}


In summary, when $m=2$, Conjecture~\ref{Conj: complete GDC} holds in the following cases:
\begin{enumerate}
    \item $1\leq r\leq \binom{2+2}{2}=6$ and $q\geq d+1$;
    \item $\binom{d+1}{2}< r\leq \binom{d+2}{2}$ and $q\geq d+1$;
    \item For $2\leq e\leq d-2$ with
    \begin{align*}
    &\tbinom{e+2}{2}<r\leq \tbinom{e+3}{2},
    &
    q&\geq \max\{d+e+\frac{e^2-1}{d-(e+1)}, d-1+e^2-e\}.
    \end{align*}
\end{enumerate}

One of our main results is that when $m=2$, Conjecture~\ref{Conj: complete GDC} is true for any $q\geq d+1$.
\begin{theorem}\label{GDC for m=2}
Given positive integers $d$, $r$ such that $1\leq r\leq \binom{d+2}{2}$ and $q\geq d+1$. We have
\[
e_r(d,2;q)=f_r(d,2;q).
\]
\end{theorem}

\subsection{Zero-dimensional conjecture}

Let $\kappa$ be an algebraically closed field.
Let $S_{d}(m,\kappa)$ be the $d^{th}$ graded component of $\kappa[x_0,\dots,x_m]$.
For $m\leq r\leq \binom{m+d}{d}$, denote
\[
u_r(d,m)
:=\max\{|V(W)| : W\subseteq S_d(m,\kappa), \dim(W)=r,\dim(V(W))=0\}.
\]

This definition is similar to $e_r(d,m;q)$, except for the requirement that $V(W)$ be of dimension zero. We need $r\geq m$, as otherwise the vanishing set of $r$ polynomials cannot be of dimension~$0$.
Also note that because $V(W)$ is zero-dimensional, we can directly take the cardinality of $V(W)$, without reference to a finite field $\fq$. In fact, in this definition, $\kappa$ can be any algebraically closed field, possibly of characteristic zero.

We will introduce some notation to state our conjecture for an exact formula of $u_r(d,m)$.
Let
\[
\Omega'(d,m)
:=\Big\{
(\alpha_1,\alpha_{2},\dots,\alpha_{m+1})\in \N_0^{m+1}
\mid \sum_{i=1}^{m+1}\alpha_i=d, 
d\notin\{\alpha_2,\dots,\alpha_{m}\}  
\Big\}.
\]
So $\Omega'(d,m)$ is obtained from $\Omega(d,m)$ by dropping $m-1$ tuples that have $d$ at the $j^{th}$ spot and $0$ everywhere else, for $2\leq j\leq m$. See Example~\ref{Eg: lower bound} for the motivation behind why these tuples are dropped.
Hence $|\Omega'(d,m)|=\binom{m+d}{d}-(m-1)$. Order $\Omega'(d,m)$ by lexicographical ordering.
Let $\omega'_r(d,m)$ be its $r^{th}$ largest element. If $\omega'_r(d,m)=(\alpha_1,\dots,\alpha_{m+1})$, we define
\[
H'_{r}(d,m)
:=\alpha_1 d^{m-1}+\dots+\alpha_m.
\]
Note that $H'_r(d,m)$ does not refer to any $q$.
\begin{conjecture}\label{Conjecture urdm}
Given $d,m\geq 1$ and $m\leq r\leq \binom{d+m}{m}$, we have
\[
u_r(d,m)=H_{r-(m-1)}'(d,m).
\]
\end{conjecture}

We show that our conjectured formula is a lower bound for $u_r(d,m)$.
\begin{proposition}\label{Prop: lower bound urdm}
For positive integers $m,d,r$ with $m\leq r\leq \binom{d+m}{m}$, we have
\[
u_r(d,m)\geq H_{r-(m-1)}'(d,m).
\]
\end{proposition}

Moreover, we prove the following special cases of Conjecture~\ref{Conjecture urdm}.

\begin{theorem}\label{thm: cases of urdm}
Given positive integers $d,m,r$ with $m\leq r\leq \binom{d+m}{m}$, if at least one of the following holds
\begin{enumerate}
    \item $d=1$;
    \item $m=1$;
    \item $m=2$;
    \item $r=m$;
    \item $\binom{m+d}{d}-d\leq r\leq \binom{m+d}{d}$,
\end{enumerate}
then we have
\[u_r(d,m)=H_{r-(m-1)}'(d,m).\]
\end{theorem}

This paper is organized as follows. We introduce some preliminaries in Section~\ref{Sec: Prelim}. In Section~\ref{Sec: lower bound}, we show that our conjectured formula is a lower bound for $u_r(d,m)$, proving Proposition~\ref{Prop: lower bound urdm}.
In Section~\ref{Sec: special cases}, we prove all parts of Theorem~\ref{thm: cases of urdm} except $m=2$. In Section~\ref{Sec: m=2 urdm} we prove the $m=2$ part of Theorem~\ref{thm: cases of urdm}. In Section~\ref{Sec: m=2 erdm} we prove Theorem~\ref{GDC for m=2}.
The proofs of some technical lemmas are deferred to Appendix~\ref{Appendix}.

\section{Preliminaries}\label{Sec: Prelim}

\subsection{Cayley-Bacharach Theorem}
Throughout this paper, $\kappa$ will denote an algebraically closed field.




\begin{definition}\label{Def: multiplicity and residual}
Let $\Gamma$ be a zero-dimensional subscheme of $\Pm(\kappa)$.
Let $I(\Gamma)$ be its homogeneous ideal, and let
\[
S(\Gamma)=\kappa[x_0,\dots,x_m]/I(\Gamma)
\]
be its graded coordinate ring.
We denote their $k^{th}$ homogeneous pieces as $I_k(\Gamma)$ and $S_k(\Gamma)$ respectively.

\begin{itemize}
    \item For $P \in \Supp(\Gamma)$, let $\mathcal O_{\Gamma,P}$ be the local ring of $\Gamma$ at $P$.
    The \emph{multiplicity} of $\Gamma$ at $P$ is
    \[
    \mult_P(\Gamma)=\dim_\kappa \mathcal O_{\Gamma,P}.
    \]
    See \cite[Definition 7.4]{hartshorne2013algebraic} for more details.

    \item We denote by $|\Gamma|_{\mm}$ the number of points of $\Gamma$ counted with multiplicity, namely
    \[
    |\Gamma|_{\mm}
    :=\sum_{P \in \Supp(\Gamma)} \mult_P(\Gamma).
    \]
    This is often called the length of the scheme $\Gamma$.
    We denote by $|\Gamma|$ the number of points of $\Gamma$ counted without multiplicity, that is,
    \[
    |\Gamma|:=|\Supp(\Gamma)|.
    \]

    \item Suppose that $\Gamma' \subseteq \Gamma$ are subschemes of $\Pm(\kappa)$.
    The \emph{residual subscheme} of $\Gamma'$ in $\Gamma$ is the subscheme $\Gamma''$ with ideal
    \[
    I(\Gamma'')
    = \Ann_{\kappa[x_0, \dots x_m]}\big(I(\Gamma')/I(\Gamma)\big).
    \]
\end{itemize}
\end{definition}


Note that while we can define the residual subscheme in general, this notion is well behaved only when $\Gamma$ is a complete intersection.
For example, we have $|\Gamma|_{\mm}= |\Gamma'|_{\mm}+|\Gamma''|_{\mm}$ when $\Gamma$ is a complete intersection (see \cite{eisenbud1996cayley}).

\begin{example}
\begin{enumerate}
    \item Suppose $I(\Gamma)=(x^2,xy,y^2)$, it is not a complete intersection. Take $I(\Gamma')=(x,y)$. Then, we have $I(\Gamma'')=(x,y)$. Note that $|\Gamma|_{\mm}=3$, $|\Gamma'|_{\mm}=1$ and $|\Gamma''|_{\mm}=1$, so $|\Gamma|_{\mm}\neq |\Gamma'|_{\mm}+|\Gamma''|_{\mm}$.
    \item Suppose $I(\Gamma)=(x^2,y^2)$, it is a complete intersection. Take $I(\Gamma')=(x,y)$. Then, we have $I(\Gamma'')=(xy,x^2,y^2)$. Note that $|\Gamma|_{\mm}=4$, $|\Gamma'|_{\mm}=1$ and $|\Gamma''|_{\mm}=3$, so $|\Gamma|_{\mm}= |\Gamma'|_{\mm}+|\Gamma''|_{\mm}$.
\end{enumerate}
\end{example}

For a zero-dimensional subscheme $\Gamma\subseteq \Pm(\kappa)$, each point $P\in\Supp(\Gamma)$ has multiplicity $1$ if and only if $\Gamma$ is reduced.
Reduced zero-dimensional subschemes of $\Pm(\kappa)$ can be identified with finite subsets of $\Pm(\kappa)$.

For a zero-dimensional subscheme $\Gamma\subseteq \Pm(\kappa)$, and nonnegative integer $k$, we denote
\[
g_{\Gamma}(k)
:=|\Gamma|_{\mm}-\dim(S_k(\Gamma)) 
=\dim(I_{k}(\Gamma))-\tbinom{m+k}{k}+|\Gamma|_{\mm}.
\]
Note that $I_{k}(\Gamma)$ is the vector space of degree $k$ polynomials that vanish on $\Gamma$. 
Next, $\dim(S_k(\Gamma))$ is the Hilbert function, giving us the dimension of the $k^{th}$ homogeneous piece of the coordinate ring, and it can also be viewed as the number of independent conditions imposed by $\Gamma$ on degree $k$ polynomials. Finally, $g_{\Gamma}(k)$ is the failure of the points in $\Gamma$ to impose independent conditions on polynomials of degree $k$.

If $X\subseteq \Pm(\kappa)$ is a finite subset, then it can be turned into a reduced zero-dimensional scheme. We see that
\[
g_{X}(k)
=\dim(I_{k}(X))-\tbinom{m+k}{k}+|X|.
\]

Noether computed the Hilbert function for the zero-dimensional subschemes of $\Pt$ that are obtained as the intersection of two curves.
\begin{proposition}\label{prop: intersection of two plane curves} \cite{noether1873ueber,eisenbud1996cayley}
Suppose $X_1$ and $X_2$ are curves in $\Pt(\kappa)$ of degrees $a$ and $b$ with no common components. Consider the zero-dimensional scheme $\Gamma=X_1\cap X_2$. Then for any $k$, the Hilbert function of the scheme $\Gamma$ can be written as
\[\dim(S_k(\Gamma))=\tbinom{k+2}{2}-\tbinom{k-a+2}{2}-\tbinom{k-b+2}{2}+ \tbinom{k-a-b+2}{2}.\]
\end{proposition}

\begin{corollary}\label{cor: g gamma curve a,b}
If we assume $1\leq a\leq b=k$, then we have
\[g_{\Gamma}(k)=\tbinom{a-1}{2}.\]
\end{corollary}

Cayley and Bacharach consider such a $\Gamma$. They further consider a subscheme $\Gamma'$ and its residual subscheme $\Gamma''$. Since $\Gamma$ is a complete intersection, we know that $ab=|\Gamma|_{\mm}=|\Gamma'|_{\mm}+|\Gamma''|_{\mm}$. Their result relates the dimension of the space of degree $k$ polynomials passing through $\Gamma'$ to the failure of $\Gamma''$ to impose independent conditions on polynomials of degree $a+b-3-k$.

\begin{theorem}\label{Thm: CB7}\cite[Theorem CB5]{eisenbud1996cayley} (Cayley-Bacharach) 
Let $X_1, X_2 \subset \Pt(\kappa)$ be plane curves of degrees $a$, $b$, respectively, such that $\Gamma=X_1 \cap X_2$ is zero-dimensional.  Let $\Gamma'$ and $\Gamma''$ be subschemes of $\Gamma$ residual to each other in $\Gamma$. 
Set $s = a+b-3$. Then, we have
\begin{equation}\label{Eq: CB5}
\dim(I_k(\Gamma'))-\dim(I_k(\Gamma))
= g_{\Gamma''}(s-k).
\end{equation}
\end{theorem}

This can be simplified when $\max(a,b)-2\leq  k$.

\begin{corollary}\label{Cor: g gamma' OP2 gamma''}
Let $\Gamma$, $\Gamma'$ and $\Gamma''$ be as above.
Assuming that $\max(a,b)-2\leq  k$, we have
\[g_{\Gamma'}(k)=\dim(I_{s-k}(\Gamma'')).\]
\end{corollary}
\begin{proof}
Applying \eqref{Eq: CB5} with $k'=s-k$ and flipping the roles of $\Gamma'$ and $\Gamma''$, we see that
\[
\dim(I_{k'}(\Gamma''))-\dim(I_{k'}(\Gamma))
= g_{\Gamma'}(s-k'),
\]
that is,
\[g_{\Gamma'}(k)=\dim(I_{s-k}(\Gamma''))-\dim(I_{s-k}(\Gamma)).\]
Moreover, Proposition~\ref{prop: intersection of two plane curves} implies that
\begin{align*}
\dim(I_{s-k}(\Gamma))
&=\tbinom{s-k+2}{2} - \dim(S_{s-k}(\Gamma))\\
&=\tbinom{s-k-a+2}{2} +\tbinom{s-k-b+2}{2} -\tbinom{s-k-a-b+2}{2}\\
&=\tbinom{b-k-1}{2} +\tbinom{a-k-1}{2} -\tbinom{-k-1}{2}
=0.
\end{align*}
The result follows.
\end{proof}

\begin{corollary}\label{Cor: gxd increasing}
Let $\Gamma$ be as in Theorem \ref{Thm: CB7}. Let $\Gamma'$ and $\Gamma_1'$ be subschemes of $\Gamma$ such that $\Gamma_1' \subseteq \Gamma'$. Then for $k \geq \max(a,b)-2$, we have
\[
g_{\Gamma_1'}(k) 
\leq g_{\Gamma'}(k)
\leq \tbinom{s-k+2}{2}.
\]
\end{corollary}
\begin{proof}
By Corollary \ref{Cor: g gamma' OP2 gamma''}, we have $g_{\Gamma'}(k)=\dim(I_{s-k}(\Gamma''))$ and $g_{\Gamma_1'}(k)=\dim(I_{s-k}(\Gamma_1''))$, where $\Gamma''$, $\Gamma_1''$ are subschemes of $\Gamma$ residual to $\Gamma'$, $\Gamma_1'$ respectively.
Since $\Gamma_1' \subseteq \Gamma'$, we have $\Gamma'' \subseteq \Gamma_1''$, so
\[
\dim(I_{s-k}(\Gamma_1''))
\leq
\dim(I_{s-k}(\Gamma''))
\leq\tbinom{s-k+2}{2}. \qedhere
\]
\end{proof}

\subsection{Relationships between conjectured polynomials}

In this subsection, we consider the relationship between $H'_{r-(m-1)}(d,m)$ and polynomials $f_r(d,m;q)$ and $H_r(d,m;q)$. In particular, we will prove Proposition~\ref{Prop: simple expression frdm}. We will also show that $H'_{r-(m-1)}(d,m)\leq H_r(d,m;d+1)$, but the proof is deferred to Appendix~\ref{Appendix}.

Given $1\leq r\leq \binom{m+d}{d}$, Conjecture~\ref{Conj: complete GDC} picks the unique $1\leq l\leq m+1$ for which
\[
\tbinom{m+d}{d}-\tbinom{m+d+1-l}{d} 
<r 
\leq \tbinom{m+d}{d}-\tbinom{m+d-l}{d}.
\]
In \cite[Proposition 4]{singhal2025conjecturebeelendattaghorpade}, the authors show that this $l$ has the property that for any $W\subseteq S_d(m)$ with $\dim(W)=r$, we have $\dim(V(W))\leq m-l$.
We will show that another way of picking the same $l$ is by looking at $\omega_r(d,m)=(\beta_1,\dots,\beta_{m+1})$ and finding the index of the first nonzero term. In order to show this we recall the relation between $r$ and the tuple $\omega_r(d,m)$.

\begin{lemma}\label{r as function of alpha}\cite[Lemma 32]{singhal2025conjecturebeelendattaghorpade}
If $\omega_r(d,m)=(\beta_1, \dots, \beta_{m+1})$, then
\begin{align*}
    r&=1+\sum_{k=1}^m \tbinom{m-k+d-\sum_{j=1}^k \beta_j}{m-k+1}.
\end{align*}
\end{lemma}

\begin{corollary}\label{Cor: index of d,000}
For $1\leq l\leq m+1$, take $r=\binom{m+d}{d}-\binom{m+d-l+1}{d}+1$, then we have $\omega_r(d,m)=(0,\dots,0,d,0,\dots,0)$, the $d$ is in the $l^{th}$ index.
\end{corollary}
\begin{proof}
Pick the $r'$ for which $\omega_{r'}(d,m)=(0,\dots,0,d,0,\dots,0)$ where $d$ is in the $l^{th}$ index. Then by Lemma~\ref{r as function of alpha} we have
\begin{align*}
r'&=1+\tbinom{m-1+d}{m}+\dots +\tbinom{m-(l-1)+d}{m-l+2}+ \tbinom{m-l+d-d}{m-l+1}+\dots +\tbinom{m-m+d-d}{m-m+1}\\
&=1+\tbinom{m-1+d}{d-1}+\dots +\tbinom{m-(l-1)+d}{d-1}\\
&=1+\tbinom{m+d}{d}-\tbinom{m-l+1+d}{d}=r. \qedhere
\end{align*}
\end{proof}

\begin{lemma}\label{Lem: first non-zero index}
Suppose $\omega_{s}(d,m)=(\alpha_1,\dots,\alpha_{m+1})$. Let $l$ be the smallest index for which $\alpha_l\neq 0$. Then we have
$$\tbinom{m+d}{d}-\tbinom{m+d+1-l}{d} <s \leq \tbinom{m+d}{d}-\tbinom{m+d-l}{d}.$$
\end{lemma}
\begin{proof}
Since $l$ is the smallest index for which $\alpha_l\neq 0$, we know that $\omega_{s}(d,m)$ is $\leq_{lex}$ to the tuple with $d$ on $l^{th}$ index and zero elsewhere. Also $\omega_{s}(d,m)$ is $>_{lex}$ to the tuple with $d$ on $(l+1)^{th}$ index and zero elsewhere. The result follows from Corollary~\ref{Cor: index of d,000}.
\end{proof}

\begin{corollary}\label{Cor: l,c}
Suppose $\omega_{s}(d,m)=(\alpha_1,\dots,\alpha_{m+1})$. Let $l$ be the smallest index for which $\alpha_l\neq 0$ and denote $\alpha_l=c$. Then we have
$$\tbinom{m+d}{d}-\tbinom{m+d+1-l}{d}+\tbinom{m+d-l-c}{d-c-1} <s \leq \tbinom{m+d}{d}-\tbinom{m+d+1-l}{d}+\tbinom{m+d-l-c+1}{d-c}.$$
\end{corollary}
\begin{proof}
This follows similarly from Lemma~\ref{r as function of alpha}.
\end{proof}

Now, we can prove Proposition~\ref{Prop: simple expression frdm}, which gives a simpler expression for $f_r(d,m;q)$.

\begin{customprop}{\ref{Prop: simple expression frdm}}
Given $1\leq r\leq \binom{m+d}{d}$, suppose $\omega_r(d,m)=(\beta_1,\dots,\beta_{m+1})$. Let $l$ be the smallest index for which $\beta_l\neq 0$. Then we have
$$\tbinom{m+d}{d}-\tbinom{m+d+1-l}{d} <r \leq \tbinom{m+d}{d}-\tbinom{m+d-l}{d},$$
and
\[f_r(d,m;q)=H_r(d,m;q)+\pi_{m-l-1}(q).\]
\end{customprop}
\begin{proof}
By Lemma~\ref{Lem: first non-zero index}, we know that $\tbinom{m+d}{d}-\tbinom{m+d+1-l}{d} <r \leq \tbinom{m+d}{d}-\tbinom{m+d-l}{d}$.
Let $j=r-\binom{m+d}{d}+\binom{m+d+1-l}{d}$. By definition, we have
$$f_r(d,m;q)=H_j(d-1,m-l+1;q) +\pi_{m-l}(q).$$
Now consider the following injective map
\[\phi:\Omega(d-1,m-l+1)\to \Omega(d,m),\]
\[\phi(\gamma_1,\dots,\gamma_{m-l+2})=(0,\dots,0,1+\gamma_1,\gamma_2,\dots,\gamma_{m-l+2}).\]
There are $l-1$ zeros on the left, so the map actually lands in $\Omega(d,m)$. Corollary~\ref{Cor: index of d,000} implies that
\[\phi(\omega_{a}(d-1,m-l+1))=\omega_{\binom{m+d}{d}-\binom{m+d-l+1}{d}+a} (d,m).\]
In particular, this means $\phi(\omega_j(d-1,m-l+1))=\omega_{r}(d,m)$. Thus,
\[H_r(d,m;q)=H_{j}(d-1,m-l+1;q) +q^{m-l}.\]
We conclude that
\begin{align*}
f_r(d,m;q)
&=H_j(d-1,m-l+1;q) +\pi_{m-l}(q)\\
&=H_r(d,m;q) -q^{m-l}+\pi_{m-l}(q)\\
&=H_r(d,m;q) +\pi_{m-l-1}(q).\qedhere
\end{align*}
\end{proof}

Next, we want to show that $H'_{r-(m-1)}(d,m)\leq H_r(d,m;d+1)$. We start by understanding the relation between the tuples $\omega_r(d,m)$ and $\omega'_{r-(m-1)(d,m)}$.

\begin{lemma}\label{Lem: w and w'}
Consider some $m\leq r< \binom{m+d}{d}$. Denote $\omega'_{r-(m-1)}(d,m)=(\alpha_1,\dots,\alpha_{m+1})$.
Pick the $s$ for which $(\alpha_1,\dots,\alpha_{m+1})=\omega_{s}(d,m)$. Let $l$ be the smallest index for which $\alpha_l\neq 0$.
Then we have $s=r-(m-l)$. If $r=\binom{m+d}{d}$, then we have $s=r$.
\end{lemma}
\begin{proof}
Note that
\[s=\#\{u\in \Omega(d,m): u\geq_{lex} (\alpha_1,\dots,\alpha_{m+1})\},\]
\[r-(m-1)=\#\{u\in \Omega'(d,m): u\geq_{lex} (\alpha_1,\dots,\alpha_{m+1})\}.\]
Therefore,
\[s-r+(m-1)=\#\{u\in \Omega(d,m)\setminus \Omega'(d,m): u\geq_{lex} (\alpha_1,\dots,\alpha_{m+1})\}.\]
The tuples in this set are precisely $(0,\dots,0,d,0,\dots,0)$ with $d$ in the $t^{th}$ index for $2\leq t\leq \min(l,m)$. Since $r\neq \binom{m+d}{d}$, we have $l\leq m$, and therefore the set has size $l-1$. This means that $s-r+(m-1)=l-1$, that is, $s=r-(m-l)$.

If $r=\binom{m+d}{d}$, then $l=m+1$. In this case $\min(l,m)=m$, so the set has size $m-1$. This means  $s-r+(m-1)=m-1$, that is, $s=r$.
\end{proof}

\begin{proposition}\label{Prop: H' smaller than H}
Given $m\leq r\leq \binom{m+d}{d}$, for every $q\geq d+1$, we have
\[H'_{r-(m-1)}(d,m)\leq H_r(d,m;q).\]
\end{proposition}
\begin{proof}
See the end of Appendix~\ref{Appendix}.
\end{proof}

\subsection{Explicit polynomials when $m=2$}

In Conjecture~\ref{Conjecture urdm}, for $m=2$, $r$ is in the range $2\leq r\leq \binom{d+2}{2}$. This range can be subdivided as 
\[\left[2,\tbinom{d+2}{2}\right]
=\bigcup_{t=3}^{d}\left(1+\tbinom{d-t+1}{2}, 1+\tbinom{d-t+2}{2}\right] 
\bigcup (1+\tbinom{d-1}{2},\tbinom{d+2}{2}].\]

\begin{lemma}\label{Lem: formula for H' m=2}
We have
\begin{equation*}
H'_{r-1}(d,2)=\begin{cases}
    td+\binom{d-t+2}{2}-r+1 &\text{ if } \binom{d-t+1}{2}<r-1\leq \binom{d-t+2}{2}, 1\leq t\leq d\\
    \binom{d+2}{2}-r &\text{ if } \binom{d-1}{2}<r-1\leq \binom{d+2}{2}-1.
\end{cases}
\end{equation*}    
\end{lemma}
Note that if $t\in \{1,2\}$ and $ \binom{d-t+1}{2}<r-1\leq \binom{d-t+2}{2}$, then both formulas are valid
\[
H'_{r-1}(d,2) 
=td+\tbinom{d-t+2}{2}-r+1
=\tbinom{d+2}{2}-r.
\]

\begin{proof}
By Lemma~\ref{Lem: w and w'}, we see that for $1\leq r-1\leq \binom{d+2}{2}-1$, we have
\begin{equation*}
H'_{r-1}(d,2)=\begin{cases}
    H_{r-1}(d,2;d) &\text{ if } 1\leq r-1\leq \binom{d+1}{2}\\
    H_{r}(d,2;d) &\text{ if } \binom{d+1}{2}<r-1\leq \binom{d+2}{2}-1.
\end{cases}
\end{equation*}
Now consider some $0\leq t\leq d$ and $\binom{d-t+1}{2}<r-1\leq \binom{d-t+2}{2}$. Let $a_1=t$ and $a_2=\binom{d-t+2}{2}-r+1$, so $0\leq a_1,a_2$ and
\[a_1+a_2=t+\tbinom{d-t+2}{2}-(r-1)\leq t+\tbinom{d-t+2}{2}- (\tbinom{d-t+1}{2}+1)=d.\]
Thus $a_3=d-a_1-a_2\geq 0$. This means that $(a_1,a_2,a_3)\in\Omega(d,2)$ and $(a_1,a_2,a_3)=\omega_s(d,2)$ for some $s$. By Lemma~\ref{r as function of alpha}, we have
\[s=1+\tbinom{1+d-a_1}{2}+\tbinom{d-a_1-a_2}{1}
=1+\tbinom{1+d-t}{2}+(d-t-\tbinom{d-t+2}{2}+r-1) =r-1.\]

Now if $t\geq 1$, then $r-1\leq \binom{d-t+2}{2}\leq \binom{d+1}{2}$, so
\[H'_{r-1}(d,2)=H_{r-1}(d,2;d)= a_1d+a_2= td+\tbinom{d-t+2}{2}-r+1.\]
Note that if $t=1$, then this simplifies as
\[H'_{r-1}(d,2)= d+\tbinom{d+1}{2}-r+1 =\tbinom{d+2}{2}-r.\]
Similarly, if $t=2$, then this simplifies as
\[H'_{r-1}(d,2)= 2d+\tbinom{d}{2}-r+1 =\tbinom{d+2}{2}-r.\]

Finally, if $t=0$, then $r-1>\binom{d-t+1}{2}=\binom{d+1}{2}$, so $H'_{r-1}(d,2)=H_{r}(d,2;d)$.
Moreover,
\[\omega_{r-1}(d,2)=(t,\tbinom{d-t+2}{2}-r+1,d-t-\tbinom{d-t+2}{2}+r-1)=(0,\tbinom{d+2}{2}-r+1,d-\tbinom{d+2}{2}+r-1),\]
so
\[\omega_{r}(d,2)=(0,\tbinom{d+2}{2}-r,d-\tbinom{d+2}{2}+r).\]
We conclude that
\[H'_{r-1}(d,2)=H_{r}(d,2;d)= \tbinom{d+2}{2}-r. \qedhere\]
\end{proof}

\section{Lower bound}\label{Sec: lower bound}

In this section, we will prove Proposition~\ref{Prop: lower bound urdm}, which shows that our conjectured formula for $u_r(d,m)$ is at least a lower bound. We will do this by explicitly constructing $r$ linearly independent polynomials whose vanishing set has size at least $H'_{r-(m-1)}(d,m)$. 
In this section, $\kappa$ continues to be an algebraically closed field.

We give an example of our construction before getting to the general proof.
\begin{example}\label{Eg: lower bound}
Suppose $m=3$ and $\omega'_{r-2}(d,m)=(\alpha_1,\alpha_2,\alpha_3,\alpha_4)$.
Pick distinct $a_1,\dots,a_d\in\kappa$, this is possible since $\kappa$ is an algebraically closed field.
Then we take
\begin{align*}
Y_1&=\{a_1,\dots,a_{\alpha_1}\}\times \{a_1,\dots,a_{d}\} \times \{a_1,\dots,a_{d}\},\\
Y_2&= \{a_d\}\times\{a_1,\dots,a_{\alpha_2}\}\times \{a_1,\dots,a_{d}\},\\
Y_3&= \{a_d\}\times \{a_d\}\times \{a_1,\dots,a_{\alpha_3}\},
\end{align*}
and $Y=Y_1\cup Y_2\cup Y_3$. It is clear that $|Y|=H'_{r-2}(d,m)$. We will show that there are $r$ linearly independent polynomials of degree $d$ that vanish on $Y$ and their vanishing set has dimension $0$. This will establish the lower bound.

Moreover, this construction motivates the definition of $\Omega'(d,3)$, which leaves out the tuples $(0,d,0,0)$ and $(0,0,d,0)$. Note that the tuple $(0,d,0,0)$ would lead to the same $Y$ (up to relabeling) as the tuple $(1,0,0,d-1)$. Similarly, the tuple $(0,0,d,0)$ would lead to the same $Y$ as the tuple $(0,1,0,d-1)$.
\end{example}

We now prove Proposition~\ref{Prop: lower bound urdm}. If $R$ is a ring of polynomials in some variables, then $R_d$ will be the subspace of homogeneous degree $d$ polynomials.

\begin{customprop}{\ref{Prop: lower bound urdm}}
For $m\leq r\leq \binom{m+d}{d}$, we have
\[u_r(d,m)\geq H_{r-(m-1)}'(d,m).\]
\end{customprop}
\begin{proof}
First, consider the case $r=\binom{m+d}{d}$. Then $H_{r-(m-1)}'(d,m)=0$, so the result is trivially true. Now suppose $m\leq r< \binom{m+d}{d}$.
We use a similar construction as the one by Heijnen and Pellikaan in \cite[IV]{heijnen1998generalized}.
Pick distinct $a_1,\dots,a_d\in\kappa$, with $a_d=0$, this is possible since $\kappa$ is an algebraically closed field. Suppose $\omega'_{r-(m-1)}(d,m)=(\alpha_1,\dots,\alpha_{m+1})$. Let $l$ be the smallest index for which $\alpha_l\neq 0$. Let $s=r-(m-l)$. So Lemma~\ref{Lem: w and w'} tells us that $(\alpha_1,\dots,\alpha_{m+1})=\omega_s(d,m)$. For $1\leq i\leq m$ denote
\begin{align*}
F_i&=\prod_{j=1}^{d}(x_i-a_jx_0),
&
\text{and}
&
&g_i=\prod_{j=1}^{\alpha_i}(x_i-a_jx_0).
\end{align*}
If $\alpha_i=0$, then $g_i=1$.

Next, consider the case $r=m$, so $(\alpha_1,\dots,\alpha_{m+1})=(d,0,\dots,0)$ and $l=1$. Then $F_1,\dots,F_m$ are linearly independent, $\dim(V(F_1,\dots,F_m))=0$ and
$$|V(F_1,\dots,F_m)|=d^m=H'_1(d,m).$$
This shows that $u_{m}(d,m)\geq H_{m-(m-1)}'(d,m)$.

For the rest of the proof, suppose $m<r<\binom{m+d}{d}$. This means $\alpha_1\neq d$ and hence $d\notin \{\alpha_1,\dots,\alpha_{m}\}$. Thus, $x_i|F_i$ and $x_i\nmid g_i$.
For $l\leq i\leq m$, denote
\begin{align*}
Y_i&=V(x_1,\dots,x_{i-1}, g_i,F_{i+1},\dots, F_{m}),
&
\text{and}
&
&Y=\bigcup_{i=l}^{m}Y_i.
\end{align*}
Notice that $|Y_i|=\alpha_i d^{m-i}$ and $|Y|=\sum_{i=l}^{m}\alpha_i d^{m-i}=H'_{r-(m-1)}(d,m)$.
Next, for $1\leq i\leq m-1$, denote
\[W_i=g_l g_{l+1}\dots g_{i}x_i \kappa[x_0,x_i,\dots,x_{m}]_{d-1-\alpha_l-\dots-\alpha_{i}}.\]
Notice that if $i<l$, then this simplifies to $W_i=x_i \kappa[x_0,x_i,\dots,x_{m}]_{d-1}$.
Also denote
\[W_m=g_l\dots g_{m} \kappa[x_0,x_{m}]_{d-\alpha_l-\dots-\alpha_{m}}.\]
Let $W=W_1+\dots+W_m+\langle F_{l+1},\dots,F_{m}\rangle$. It is clear that $Y\subseteq V(W)$, so
\[|V(W)|\geq |Y|=H'_{r-(m-1)}(d,m).\]
Moreover, $V(W)\subseteq V(F_1,\dots,F_m)$, so $\dim(V(W))=0$.
We want to compute $\dim(W)$. For this, we want to show that the sum used to define $W$ is a direct sum.

Suppose for $1\leq i\leq m-1$, we have $h_i\in \kappa[x_0,x_i,\dots,x_{m}]_{d-1-\alpha_1-\dots-\alpha_i}$ and $h_m\in \kappa[x_0,x_m]_{d-\alpha_1-\dots-\alpha_m}$ and $u_{l+1},\dots,u_m\in \kappa$ such that
\[\sum_{i=1}^{l}\Big( x_i\prod_{j=1}^{i}g_j\Big)h_i+
\sum_{i=l+1}^{m-1}\Big(\Big( x_i\prod_{j=1}^{i}g_j\Big)h_i +u_{i}F_i\Big) +\Big(\prod_{j=1}^{m}g_j\Big)h_m +u_{m}F_m=0.\]
We will show that all $h_i$ and $u_i$ are zero, which will show that it is a direct sum.
Assume for the sake of contradiction that some $h_i$ or $u_i$ is non-zero. Consider the smallest such index, denote it as $i_0$.
\begin{enumerate}
    \item Case 1: $1\leq i_0\leq l$. We know that
    \[\sum_{i=i_0}^{l}\Big( x_i\prod_{j=1}^{i}g_j\Big)h_i+
\sum_{i=l+1}^{m-1}\Big(\Big( x_i\prod_{j=1}^{i}g_j\Big)h_i +u_{i}F_i\Big) +\Big(\prod_{j=1}^{m}g_j\Big)h_m +u_{m}F_m=0.\]
Note that for $i>i_0$, the terms do not involve $x_{i_0}$. Therefore if we plug in $x_{i_0}=0$ and subtract the resulting equation from the equation above, it implies that $( x_{i_0}\prod_{j=1}^{i_0}g_j)h_{i_0} =0$. This means $h_{i_0}=0$ which is a contradiction.

\item Case $2$: $l+1\leq i_0\leq m-1$. We know that
    \[
\sum_{i=i_0}^{m-1}\Big(\Big( x_i\prod_{j=1}^{i}g_j\Big)h_i +u_{i}F_i\Big) +\Big(\prod_{j=1}^{m}g_j\Big)h_m +u_{m}F_m=0.\]
Again note that for $i>i_0$, the terms do not involve $x_{i_0}$. Therefore if we plug in $x_{i_0}=0$ and subtract the resulting equation from the equation above, it implies that $( x_{i_0}\prod_{j=1}^{i_0}g_j)h_{i_0} +u_{i_0}F_{i_0} =0$.
Since at least one of $h_{i_0}$ or $u_{i_0}$ is nonzero, it follows that they both must be nonzero.
Now $\alpha_l>0$ and $i_0>l$ imply that
\begin{align*}
&(x_l-a_1 x_0)\mid g_l,
&
&g_l\mid \prod_{j=1}^{i_0} g_j,
&
&\prod_{j=1}^{i_0} g_j \mid F_{i_0}.
\end{align*}
This means $(x_l-a_1 x_0)\mid F_{i_0}$. Since $i_0>l$, this contradicts the definition of $F_{i_0}$.

\item Case 3: $i_0>l$ and $i_0=m$. We know that
    \[ \Big(\prod_{j=1}^{m}g_j\Big)h_m +u_{m}F_m=0.\]
    Since at least one of $h_{m}$ and $u_{m}$ is nonzero, it follows that they both must be nonzero. Arguing as in Case 2, we see that $(x_l-a_1 x_0)\mid F_{m}$. Since $m>l$, this contradicts the definition of $F_{m}$.
\end{enumerate}
It follows that
\[W=W_1\oplus\dots\oplus W_m\oplus\langle F_{l+1},\dots,F_{m}\rangle.\]
Note that for $1\leq i\leq m-1$, $\dim(W_i)=\binom{m-(i-1)+d-1-\sum_{j=1}^{i}\alpha_j}{m-(i-1)}$ and $\dim(W_m)=\binom{1+d-\sum_{j=1}^{m}\alpha_j}{1}$.
Therefore, by Lemma~\ref{r as function of alpha}, we have
\begin{align*}
\dim(W)
&=(m-l)+\sum_{i=1}^{m-1} \tbinom{m-(i-1)+d-1-\sum_{j=1}^{i}\alpha_j}{m-(i-1)} +\tbinom{1+d-\sum_{j=1}^{m}\alpha_j}{1}\\
&=1+(m-l)+\sum_{i=1}^{m} \tbinom{m-i+d-\sum_{j=1}^{i}\alpha_j}{m-(i-1)}
=(m-l)+s
=r.
\end{align*}
We have shown that $\dim(W)=r$, $\dim(V(W))=0$ and $|V(W)|\geq H'_{r-(m-1)}(d,m)$. It follows that $u_r(d,m)\geq H'_{r-(m-1)}(d,m)$.
\end{proof}

\section{Special Cases of Conjecture~\ref{Conjecture urdm}}\label{Sec: special cases}

In this section, we will prove all the parts of Theorem~\ref{thm: cases of urdm} except $m=2$.
We start by proving a lemma that when $X$ is a small set of points, then each point imposes a new condition on polynomials of degree $d$, and hence $g_X(d)=0$.

\begin{lemma}
Suppose $X\subseteq \PPP^m$ is a finite subset with $|X|\leq d$ and $P\in \Pm\setminus X$. Then
\[
\dim(I_d(X\cup \{P\}))
=\dim(I_d(X))-1.
\]
\end{lemma}
\begin{proof}
Consider the map from $I_d(X)$ to $\kappa$ given by evaluating at $P$. Since $I_d(X\cup \{P\})$ is the kernel of this map, we see that its dimension is either $\dim(I_d(X))$ or $\dim(I_d(X))-1$. In order to show that it is $\dim(I_d(X))-1$, we need to construct a polynomial in $I_d(X)$ that does not vanish at $P$. For each point in $X$, we can find a degree one polynomial that vanishes at that point but not on $P$. By multiplying them, we get a polynomial of degree $|X|$ that vanishes on $X$ but not on $P$. Since $|X|\leq d$, we can multiply this by a polynomial of degree $d-|X|$ that does not vanish on $P$, and obtain a polynomial in $I_d(X)$ that does not vanish at $P$. The result follows.
\end{proof}

\begin{corollary}\label{Cor: gx is 0}
Suppose $X\subseteq \PPP^m$ is a finite subset with $|X|\leq d+1$. Then $g_X(d)=0$.    
\end{corollary}
\begin{proof}
The previous lemma implies that $\dim(I_d(X))=\binom{m+d}{d}-|X|$. It follows that $g_X(d)=0$.
\end{proof}

We can now prove the part of Theorem~\ref{thm: cases of urdm} where $r$ satisfies $\binom{m+d}{d}-d\leq r\leq \binom{m+d}{d}$.

\begin{proposition}
Given $m,d\geq 1$ and $\binom{m+d}{d}-d\leq r\leq \binom{m+d}{d}$, we have
\[u_r(d,m)=H_{r-(m-1)}'(d,m).\]    
\end{proposition}
\begin{proof}
For $r$ in this range, $H'_{r-(m-1)}(d,m)=H_r(d,m;d)=\binom{m+d}{d}-r$.
From Proposition~\ref{Prop: lower bound urdm}, we know that $u_r(d,m)\geq H_{r-(m-1)}'(d,m)$.
Consider $W\subseteq S_d(m,\kappa)$ with $\dim(W)=r$ and $\dim(V(W))=0$. 

Assume for the sake of contradiction that $|V(W)|>\binom{m+d}{d}-r$. Pick a subset $X\subseteq V(W)$ of size $|X|=\binom{m+d}{d}-r+1$. Notice that
\[|X|\leq \tbinom{m+d}{d} -(\tbinom{m+d}{d}-d)+1=d+1.\]
Thus Corollary~\ref{Cor: gx is 0} says that $g_X(d)=0$, that is, $\dim(I_d(X))=\binom{m+d}{d}-|X|$. Therefore,
\[r= \dim(W)\leq \dim(I_d(X))= \tbinom{m+d}{d}-|X|=r-1.\]
This is a contradiction, and hence $|V(W)|\leq\binom{m+d}{d}-r$.

We conclude that $u_r(d,m)\leq H_{r-(m-1)}'(d,m)$ which completes the proof.
\end{proof}

If $d=1$ or $m=1$, then every $r$ in the range $m\leq r\leq \binom{m+d}{d}$, satisfies $\binom{m+d}{d}-d\leq r\leq \binom{m+d}{d}$.
Therefore, the cases $d=1$ and $m=1$ of Theorem~\ref{thm: cases of urdm} follow as corollaries.

\begin{corollary}[d=1]
Given $m\geq 1$ and $m\leq r\leq \binom{m+1}{1}$, we have
\[u_r(1,m)=H_{r-(m-1)}'(1,m).\]
\end{corollary}
\begin{corollary}[m=1]
Given $d\geq 1$ and $1\leq r\leq \binom{d+1}{1}$, we have
\[u_r(d,1)=H_{r}'(d,1).\]
\end{corollary}

Next, we want to prove the case $r=m$ of Theorem~\ref{thm: cases of urdm}. We first recall a result from \cite{singhal2025conjecturebeelendattaghorpade}. Let $\deg_k(X)$ be the sum of the degrees of all $k$-dimensional components of $X$.

\begin{proposition}\cite[Proposition 13]{singhal2025conjecturebeelendattaghorpade}\label{Prop: bound on the weighted sum}
Suppose $F_1,\dots,F_r\in \kappa[x_0,\dots,x_m]$ are homogeneous polynomials of degree at most $d$ and let $X=V(F_1,\dots,F_r)$. Then we have
$$\sum_{j=1}^{m} d^{-j}\deg_{m-j}(X) \leq 1.$$
\end{proposition}

\begin{proposition}[r=m]
Given $m,d\geq 1$, we have
\[u_m(d,m)=H_{1}'(d,m).\]
\end{proposition}
\begin{proof}
Notice that $H'_{1}(d,m)=d^m$.
From Proposition~\ref{Prop: lower bound urdm}, we know that $u_m(d,m)\geq H_{1}'(d,m)$.
Consider $W\subseteq S_d(m,\kappa)$ with $\dim(W)=m$ and $\dim(V(W))=0$. Proposition~\ref{Prop: bound on the weighted sum} says that $d^{-m}\deg_0(V(W))\leq 1$, meaning $|V(W)|\leq d^m$.  We conclude that $u_m(d,m)\leq H_{1}'(d,m)$ which completes the proof.
\end{proof}

\section{Computing $u_r(d,2)$}\label{Sec: m=2 urdm}

In this section, our goal is to prove the $m=2$ case of Theorem~\ref{thm: cases of urdm}. We will consider vector subspaces $W\subseteq S_d(2,\kappa)$ with $\dim(W)=r$ and $\gcd(W)=1$. This is because $\gcd(W)=1$ is equivalent to $\dim(V(W))=0$.
Since we have proved Proposition~\ref{Prop: lower bound urdm}, we need to show that for each such $W$, $|V(W)|\leq H_{r-1}'(d,2)$. We will in fact show that $|V(W)|_{\mm}\leq H_{r-1}'(d,2)$, which is stronger.
In this section, $\kappa$ continues to be an algebraically closed field.

We start with a lemma about such $W$.


\begin{lemma}\label{Lem: 2 coprime poly}
Let $W$ be a vector subspace of $\kappa[x_0,x_1,\dots,x_m]_{d}$. If $\gcd(W)=1$, then there is a pair of coprime polynomials in $W$.
\end{lemma}
\begin{proof}
Pick a non-zero polynomial $F\in W$. Suppose that its irreducible factors are $f_1,\dots,f_t$. 

For each $1\leq i\leq t$, we know that there is a polynomial $F_i\in W$ that is not divisible by $f_i$.
For $1\leq i\leq t$, consider the map $\phi_i:\kappa^t\to \kappa[x_0,x_1,\dots,x_m]/f_i$ defined as
\[\phi_i(a_1,\dots,a_t)=\sum_{j=1}^t a_jF_j \pmod{f_i}.\]
Since $F_i\not\equiv0\pmod{f_i}$, we know that $\ker(\phi_i)\neq \kappa^t$. Since $\ker(\phi_i)$ is a vector subspace in $\kappa^t$, this means that $\dim(\ker(\phi_i))\leq t-1$. Since this holds for each $1\leq i\leq t$ and $\kappa$ is an algebraically closed field, we see that $\bigcup_{i=1}^{t}\ker(\phi_i) \neq \kappa^t$.

Pick $(a_1,\dots,a_t)\in\kappa^t$ which is not in $\bigcup_{i=1}^{t} \ker(\phi_i)$. Then $G=\sum_j a_j F_j$ is not divisible by any $f_i$. Thus, $\gcd(F,G)=1$.
\end{proof}

Recall that for $m=2$, Conjecture~\ref{Conjecture urdm} considers $r$ in the range $2\leq r\leq \binom{d+2}{2}$. This range can be subdivided as 
\[\left[2,\tbinom{d+2}{2}\right]
=\bigcup_{t=3}^{d}\left(1+\tbinom{d-t+1}{2}, 1+\tbinom{d-t+2}{2}\right] 
\bigcup (1+\tbinom{d-1}{2},\tbinom{d+2}{2}].\]
We first consider $r$ in the range $(1+\tbinom{d-1}{2},\tbinom{d+2}{2}]$. For this, we will need a lemma similar to Corollary~\ref{Cor: gx is 0}, saying that a small number of points impose independent conditions on polynomials of degree $d$.

\begin{lemma}\label{Lem: 3d-1 independent cond}
Suppose $G_1$ and $G_2$ are coprime polynomials in $S_d(2,\kappa)$ and $X\subseteq V(G_1,G_2)$ is a subscheme with $|X|_{\mm}\leq 3d-1$. Then
\[g_X(d)=0.\]
\end{lemma}
\begin{proof}
Let $\Gamma=V(G_1,G_2)$. Since $G_1$ and $G_2$ are coprime, $\Gamma$ is zero-dimensional and a complete intersection.
Let $X'$ be the residual subscheme of $X$ in $\Gamma$.
By Corollary~\ref{Cor: g gamma' OP2 gamma''}, we know that
\[g_X(d)=\dim(I_{d-3}(X')).\]
Note that
\[ |X'|_{\mm}=|\Gamma|_{\mm}-|X|_{\mm}\geq  d^2-(3d-1)= (d-3)d+1.\]
Consider a polynomial $h$ of degree $d-3$. Suppose it factors as $h=h_1\dots h_n$. Then each $h_i$ must be coprime to at least one of $G_1$ or $G_2$. 
If it is coprime to $G_1$, then $|V(h_i)\cap \Gamma|_{\mm}\leq |V(h_i)\cap V(G_1)|_{\mm}=\deg(h_i)d$. Similarly, if it were coprime to $G_2$, then we will still have $|V(h_i)\cap \Gamma|_{\mm}\leq \deg(h_i)d$.
Thus
\[|V(h)\cap \Gamma|_{\mm}\leq \sum_{i=1}^{n} |V(h_i)\cap \Gamma|_{\mm}\leq \sum_{i=1}^{n}\deg(h_i)d =(d-3)d.\]
This means that $X'\not\subseteq V(h)$, so $\dim(I_{d-3}(X'))=0$. This completes the proof.
\end{proof}

\begin{lemma}\label{Lem: Xm in VW}
Suppose $Y$ is a zero-dimensional subscheme of $\Pm(\kappa)$. Then there is a further subscheme $X\subseteq Y$, such that $|X|_{\mm}=|Y|_{\mm}-1$.
\end{lemma}
\begin{proof}
Let $H$ be a linear hyperplane such that $Y \cap H = \emptyset$. By a change of coordinate, we can assume that $H=V(x_0)$. Then, $Y$ is in the non-vanishing locus of $x_0$. We can de-homogenize with respect to $x_0$ and consider the ideal $I \subseteq \kappa[y_1,\dots, y_m]$ which is the dehomogenization of $I(Y)$.
Denote $R=\kappa[y_1,\dots,y_m]/I$, so $Y=\Spec(R)$ and $|Y|_{\mm}=\dim_{\kappa}(R)$.
For each closed point $P$, denote the maximal ideal by $\m_P\subseteq R$ and denote the local rings as $R_P=\OO_{Y,P} \cong (\kappa[y_1,\dots, y_m]/I)_{\m_P}$.
Since $R$ is Artinian, we know that $R= \prod_{P\in\Supp(Y)} R_P$.

Pick a point $P\in\Supp(Y)$, let $l$ be the multiplicity of $Y$ at $P$.
By Definition \ref{Def: multiplicity and residual}, we know that $l=\dim_{\kappa}R_P$.
Since $\kappa$ is algebraically closed, we have $R_P/\m_P \cong \kappa$.  
Since $R_P$ is a local Artinian ring, its maximal ideal $\m_P$ is nilpotent. Hence, there exists $t$ such that $\m_P^t \neq 0 $ and $\m_P^{t+1} =0$. 
Pick $u \in \m_P^t$ such that $u \neq 0$. Then,
\[
\dim_{\kappa}(R_P/(u))
=\dim_\kappa(R_P)-\dim_\kappa(uR_P).
\]
We claim that $\dim_\kappa(uR_P)=1$. Since $u \neq 0$, we know that $\dim_\kappa(uR_P) \geq 1$. 
On the other hand, multiplication by $u$ annihilates $\m_P$, so $uR_P$ is a quotient of $R_P/\mathfrak{m}_P \cong \kappa$, so $\dim_\kappa(uR_P)=1$. Therefore, we get that $\dim_{\kappa}(R_P/(u))=l-1$.

Since $R= \prod_{Q\in\Supp(Y)} R_Q$, we can pick $\tilde{u}\in R$ such that $\tilde{u}=(u,0,\dots,0)$, where $u$ is at $P$ and zeros at all other points.
Recall that $R=\kappa[y_1,\dots,y_m]/I$ and let $f\in \kappa[y_1,\dots,y_m]$ be a lift of $\tilde{u}$.

Let $X=V(I+(f))$, it is a subscheme of $Y$.
Moreover, $X=\Spec(R/(\tilde{u}))$ and
\[
R/(\tilde{u})
=R_P/(u)
\times \prod_{\substack{Q\in\Supp(Y)\\ Q\neq P} R_Q}.
\]
So $\mult_P(X)=\mult_P(Y)-1$ and for other points $Q$, $\mult_Q(X)=\mult_Q(Y)$.
We conclude that $|X|_{\mm}=|Y|_{\mm}-1$.
%
%
%
\end{proof}

We are now ready to prove the $m=2$ case of Theorem~\ref{thm: cases of urdm} with the additional requirement that $1+\binom{d-1}{2}<r\leq \binom{d+2}{2}$.
\begin{proposition}\label{Prop: udr Case 2}
If $1+\binom{d-1}{2}<r\leq \binom{d+2}{2}$, then
\[u_r(d,2)=H'_{r-1}(d,2)=\tbinom{d+2}{2}-r.\]
\end{proposition}
\begin{proof}
By Lemma~\ref{Lem: formula for H' m=2} and Proposition~\ref{Prop: lower bound urdm} we only need to show that $u_r(d,2)\leq \binom{d+2}{2}-r$. Assume for the sake of contradiction that $u_r(d,2)> \binom{d+2}{2}-r$. This means that there is some $W\subseteq S_d(2,\kappa)$, with $\dim(W)=r$, $\gcd(W)=1$ and 
\[
|V(W)|_{\mm}\geq
\tbinom{d+2}{2}-r+1.
\] 

By Lemma \ref{Lem: Xm in VW}, we can pick a subscheme $X\subseteq V(W)$ with 
\[
|X|_{\mm}
=\tbinom{d+2}{2}-r+1.
\] 
Since $W\subseteq I(X)$, we know that  $\dim(I_d(X)) \geq r$.

By Lemma~\ref{Lem: 2 coprime poly}, we know that there are two coprime polynomials $G_1, G_2\in W$. Clearly $X\subseteq V(G_1,G_2)$.
Next, notice that
\[
|X|_{\mm}
\leq \tbinom{d+2}{2} -\Big(\tbinom{d-1}{2}+2\Big) +1
=3d-1.
\]
By Lemma~\ref{Lem: 3d-1 independent cond}, we know that $g_{X}(d)=0$ and therefore
\[r\leq \dim(I_d(X)) =\tbinom{d+2}{2}-|X|_{\mm}+0=r-1.\]
This is a contradiction and hence $|V(W)|_{\mm}\leq
\tbinom{d+2}{2}-r$.
\end{proof}

Next, we consider the ranges $\left(1+\binom{d-t+1}{2}, 1+\binom{d-t+2}{2}\right]$ for $3\leq t\leq d$.
In order to prove the $m=2$ case of Theorem~\ref{thm: cases of urdm} with these ranges, we will show an analog of Lemma~\ref{Lem: 3d-1 independent cond}. This is, if a subscheme $X\subseteq V(G_1,G_2)$ does not have too many points, then $g_X(d)$ is small.
We first show that the bound on $g_X(d)$ implies $u_r(d,2)\leq H'_{r-1}(d,2)$ for $r$ in this range.

\begin{lemma}\label{Lem: cond on gx implies udr}
Consider positive integers $t$, $d$ satisfying $3\leq t\leq d$.

If for any coprime polynomials $G_1$, $G_2$ in $S_d(2,\kappa)$ and for any subscheme $X\subseteq V(G_1,G_2)$, with $|X|_{\mm}\leq td+d-t+1$, it is known that
\[g_X(d)\leq \tbinom{t-1}{2}.\]
Then, for each $r$ in the range
\[1+\tbinom{d-t+1}{2} <r\leq 1+\tbinom{d-t+2}{2},\]
we have
\[u_r(d,2)\leq td+1+\tbinom{d-t+2}{2}-r.\]
\end{lemma}
\begin{proof}
Assume for the sake of contradiction that $u_r(d,2)> td+1+\binom{d-t+2}{2}-r$. This means that there is some $W\subseteq S_d(2,\kappa)$ with $\dim(W)=r$, $\gcd(W)=1$ and 
\[
|V(W)|_{\mm}
\geq td+2+\tbinom{d-t+2}{2}-r.
\]
By Lemma~\ref{Lem: Xm in VW}, we can pick a subscheme $X_1\subseteq V(W)$, for which 
\[
|X_1|_{\mm}
=td+2+\tbinom{d-t+2}{2}-r.
\]
Note that
\[
|X_1|_{\mm} 
\leq td+2+\tbinom{d-t+2}{2} -\big(2+\tbinom{d-t+1}{2}\big) 
=td+d-t+1.
\]
From Lemma~\ref{Lem: 2 coprime poly}, we know that there are coprime polynomials $G_1$, $G_2$ in $W$. This means $X_1\subseteq V(G_1,G_2)$ and $|X_1|_{\mm}\leq td+d-t+1$. Thus, by the assumption in the statement, we know that $g_{X_1}(d)\leq \binom{t-1}{2}$. We see that
\begin{align*}
r\leq &\dim(I_{d}(V(W))) \leq \dim(I_d(X_1))
=\tbinom{d+2}{2} -|X_1|_{\mm} +g_{X_1}(d)\\
&\leq \tbinom{d+2}{2} -\big(td+2+\tbinom{d-t+2}{2}-r\big) +\tbinom{t-1}{2}\\
&=r-1.
\end{align*}
This is a contradiction.
\end{proof}

We will now prove by induction on $d$ that if $|X|_{\mm}\leq td+d-t+1$, then $g_X(d)\leq \binom{t-1}{2}$. By the induction hypothesis, we know that this holds for $d'<d$. Hence, by Lemma~\ref{Lem: cond on gx implies udr} and Proposition~\ref{Prop: udr Case 2}, we know that $u_{r'}(d',2)\leq H'_{r'-1}(d',2)$.

\begin{proposition}\label{Prop: cond on gx udr}
Consider positive integers $t$, $d$ satisfying $3\leq t\leq d$.
Given coprime polynomials $G_1$, $G_2$ in $S_d(2,\kappa)$ and a subscheme $X\subseteq V(G_1,G_2)$, with $|X|_{\mm}\leq td+d-t+1$, we have
\[g_X(d)\leq \tbinom{t-1}{2}.\]
\end{proposition}
\begin{proof}
We prove this by induction on $d$. Assume that the statement is known for smaller values of $d$.
Let $\Gamma=V(G_1,G_2)$. First, if $t=d$, then it follows from Corollary~\ref{Cor: gxd increasing} that
\[g_{X}(d)\leq g_{\Gamma}(d)=\tbinom{d-1}{2}.\]
This finishes the base case $d=3$, since $d=3$ forces $t=d$. This also deals with the $t=d$ case for larger $d$.

Now suppose $3\leq t\leq d-1$, so $d\geq 4$.
Let $X'$ be the residual subscheme of $X$ in $\Gamma$. So
\[|X'|_{\mm}=d^2-|X|_{\mm}\geq d^2-dt -d+t-1.\]
By Lemma \ref{Lem: Xm in VW}, we can pick a subscheme $Y\subseteq X'$ such that $|Y|_{\mm}= d^2-dt-d+t-1$.
Corollary~\ref{Cor: g gamma' OP2 gamma''} implies that $g_{X}(d)=\dim(I_{d-3}(X'))$. This means $g_{X}(d)\leq \dim(I_{d-3}(Y))$.

Assume for the sake of contradiction that $g_{X}(d)\geq \tbinom{t-1}{2}+1$.
This implies
\[\dim(I_{d-3}(Y))\geq \tbinom{t-1}{2}+1.\]
From here we want to bound $|Y|_{\mm}$ and derive a contradiction.
Denote
\begin{align*}
r_1&=\dim(I_{d-3}(Y)),
&
g&=\gcd(I_{d-3}(Y)),
&
e&=\deg(g).
\end{align*}
It is possible that $e=0$, if $g=1$. Say $I_{d-3}(Y)=gW$, so $\gcd(W)=1$ and
\[\dim(W)=\dim(I_{d-3}(Y))=r_1\geq \tbinom{t-1}{2}+1.\]
Let $d_1=d-3-e$, so the polynomials in $W$ are of degree $d_1$.

Now $I_{d-3}(Y)=gW$, so $Y\subseteq V(g)\cup V(W)$, and hence
\[Y\subseteq \big(V(g)\cap \Gamma \big)\cup V(W).\]
We will bound $|Y|_{\mm}$ by bounding $|V(g)\cap \Gamma|_{\mm}$ and $|V(W)|_{\mm}$ separately.
Notice that
\[d(d-t)-d+t-1 = |Y|_{\mm} \leq |V(g)\cap \Gamma|_{\mm} + |V(W)|_{\mm}\leq |V(g)\cap \Gamma|_{\mm} +u_{r_1}(d_1,2).\]

We first bound $|V(g)\cap \Gamma|_{\mm}$.
Suppose $g$ factors as $g=g_1\dots g_u$. Then for each $g_i$, it does not divide at least one of $G_1$ or $G_2$ (since $\gcd(G_1,G_2)=1$). Thus, 
\[
|V(g_i)\cap V(G_1,G_2)|_{\mm}
\leq \deg(g_i)d.
\] 
It follows that $|V(g)\cap \Gamma|_{\mm}\leq ed$.

Next, we will bound $u_{r_1}(d_1,2)$. We will use the induction hypothesis as $d_1<d$.
Note that
\[1+\tbinom{t-1}{2}\leq r_1=\dim(I_{d-3}(Y)) =\dim(W)\leq \dim(\kappa[x_0,x_1,x_2]_{d-3-e}).\]
This implies $1+\tbinom{t-1}{2}\leq \binom{d-1-e}{2}$ and hence $e<d-t$.
Next,
\[|V(W)|_{\mm}\leq u_{r_1}(d_1,2)\leq u_{\binom{t-1}{2}+1}(d_1,2).\]
Let $t_1=d_1-t+3=d-e-t$, so $t-1= d_1-t_1+2$ and
\[1+\tbinom{d_1-t_1+1}{2}< \tbinom{t-1}{2}+1 =\tbinom{d_1-t_1+2}{2}+1.\]
The condition $e<d-t$ means $t_1>0$.

\begin{enumerate}
    \item Case 1: $e\leq d-t-3$. This means $t_1\geq 3$.
    We know that $d_1<d$, so the statement is known for $d_1$ by the induction hypothesis. Then Lemma \ref{Lem: cond on gx implies udr} says that
    \[u_{\binom{t-1}{2}+1}(d_1,2) \leq t_1 d_1+1+\tbinom{d_1-t_1+2}{2} -(\tbinom{t-1}{2}+1)=t_1d_1.\]

    \item Case 2: $d-t-2 \leq e\leq d-t-1$, so $1\leq t_1\leq 2$. Then Proposition~\ref{Prop: udr Case 2} implies that
    \[u_{\binom{t-1}{2}+1}(d_1,2)=\tbinom{d_1+2}{2}-\tbinom{t-1}{2}-1 = \tbinom{d_1+2}{2}-\tbinom{d_1-t_1+2}{2}-1.\]
    Since $t_1\in\{1,2\}$, this simplifies to $t_1d_1$.
\end{enumerate}
Therefore in both cases we have 
\[u_{\binom{t-1}{2}+1}(d_1,2)\leq t_1d_1.\]
We see that
\begin{align*}
d^2-dt-d+t-1 &= |Y|_{\mm} \leq ed+ |V(W)|_{\mm}\\
&\leq ed+ (d-e-t) (d-e-3)\\
&= e^2 -(d-t-3)e +(d-t)(d-3).
\end{align*}
Since $0\leq e\leq  d-t-1$, by the shape of the parabola, we see that 
\[e^2 -(d-t-3)e\leq 2(d-t-1).\]
and hence
\[d^2-dt-d+t-1\leq 2(d-t-1) +(d-t)(d-3).\]
This simplifies to 
\[d^2-dt-d+t-1\leq d^2-dt-d+t-2.\]
which is a contradiction.
\end{proof}

\begin{theorem}\label{Thm: urd2}
Given $d\geq 1$ and $2\leq r\leq \binom{d+2}{2}$, we have
\[u_r(d,2)=H_{r-1}'(d,2).\]
\end{theorem}
\begin{proof}
This follows from Proposition~\ref{Prop: udr Case 2}, Proposition~\ref{Prop: lower bound urdm}, Lemma~\ref{Lem: formula for H' m=2}, Lemma~\ref{Lem: cond on gx implies udr} and Proposition~\ref{Prop: cond on gx udr}.
\end{proof}

This completes the proof of Theorem~\ref{thm: cases of urdm}.

\section{Computing $e_r(d,2;q)$}\label{Sec: m=2 erdm}

In this section, our goal is to prove Theorem~\ref{GDC for m=2}. First, we recall a result of Beelen, Datta, and Ghorpade \cite{beelen2022combinatorial} which states that their conjectured formula $f_r(d,m;q)$ is at least a lower bound for $e_r(d,m;q)$.
In this section, we will work over a finite field $\fq$.

\begin{proposition}\cite[Theorem 2.3]{beelen2022combinatorial}\label{Lower bound for erdm}\\
Suppose we are given $1 \leq r \leq \binom{m+d}{d}$ and $q\geq d+1$. Then we have
$$e_r(d,m;q)\geq f_r(d,m;q).$$
\end{proposition}

Therefore, we only need to prove that $e_r(d,2;q)\leq f_r(d,2;q)$. This means that, given a subspace $W$ of $S_d(2,\fq)$ of dimension $r$, we want to show that $|V(W)(\fq)|\leq f_r(d,2;q)$.
We decompose $V(W)$ as $X_1\cup X_0$, where $X_1$ has all the one-dimensional components of $V(W)$ and $X_0$ has the zero-dimensional components. Then $X_1=V(g)$ where $g=\gcd(W)$, so
\[
|X_1(\fq)|
\leq e_1(\deg(g),2)
=f_1(\deg(g),2)
=\deg(g) q+1.
\]
Moreover, we can bound $|X_0|$ using Theorem~\ref{Thm: urd2}. We recall a lemma from \cite{singhal2025conjecturebeelendattaghorpade}, which allows us to combine the two bounds.

\begin{lemma}\cite[Lemma 6.8]{singhal2025conjecturebeelendattaghorpade}\label{Lem d-c to d}
Given $m\geq 1$, $1\leq c\leq d-1$ and $1\leq r\leq \binom{m+d-c}{d-c}$, we have
$$H_r(d-c,m;q)+cq^{m-1}\leq H_r(d,m;q).$$
\end{lemma}

We are now ready to prove Theorem~\ref{GDC for m=2}.

\begin{customthm}{\ref{GDC for m=2}}
For $d\geq 1$, $1\leq r\leq \binom{d+2}{2}$ and $q\geq d+1$, we have
\[e_r(d,2;q)=f_r(d,2;q).\]
\end{customthm}
\begin{proof}
By Proposition \ref{Lower bound for erdm}, it is already known that $e_r(d,2;q)\geq f_r(d,2;q)$. Now, consider $W\subseteq S_d(2)$ of dimension $r$. Say $\gcd(W)=g$ and $W=gW_1$.
\begin{enumerate}
\item Case 1: $\deg(g)=0$, that is, $\gcd(W)=1$. Then it follows from Theorem~\ref{Thm: urd2}, Proposition~\ref{Prop: simple expression frdm} and Proposition~\ref{Prop: H' smaller than H} that
\[
|V(W)(\fq)|
\leq
|V(W)|_{\mm}
\leq H'_{r-1}(d,2)
\leq H_{r}(d,2;q)
\leq f_r(d,2;q).
\]
\item Case 2: $\deg(g)\geq 1$. Say $e=\deg(g)$. We still have $\gcd(W_1)=1$, so Theorem~\ref{Thm: urd2} and Proposition~\ref{Prop: H' smaller than H} imply that
\[
|V(W_1)(\fq)|
\leq |V(W_1)|_{\mm}
\leq H'_{r-1}(d-e,2)
\leq H_r(d-e,2;q).
\]
Moreover, 
\[
|V(g)(\fq)|
\leq e_1(e,2)
=eq+1.
\]
Since $\dim(W_1)=r$ and $W_1\subseteq \fq[x_0,x_1,x_2]_{d-e}$, we know that 
\[
r\leq 
\tbinom{d-e+2}{2} \leq
\tbinom{d-1+2}{2},
\]
so $l=1$.
Thus, Lemma \ref{Lem d-c to d} implies that 
\begin{align*}
|V(W)(\fq)|
&\leq |V(W_1)(\fq)|+|V(g)(\fq)|
\leq H_r(d-e,2;q) +eq +1\\
&\leq H_r(d,2;q)+1
= H_r(d,2;q)+\pi_{2-l-1}(q)
=f_r(d,2;q). \qedhere
\end{align*}
\end{enumerate}
\end{proof}

\section{Conjecture of Bogulavsky}

We start by recalling a conjecture of Bogulavsky regarding the degrees in different dimensions of $V(W)$.
In this section, we will again work over a finite field $\fq$.

\begin{conjecture}\cite[Conjecture 2]{boguslavsky1997number}\label{Conj: Bogulavsky}
Suppose we have $m,d\geq 1$ and $1\leq r\leq \tbinom{m+d}{d}$. Denote $\omega_r(d,m)=(\beta_1,\dots,\beta_{m+1})$. Then, given a subspace $W$ of $S_d(m,\fq)$, we have
\[
\Big(\deg_{m-1}(V(W)),
\dots,\deg_0(V(W)),
d-\sum_{i=0}^{m-1} \deg_i(V(W))
\Big) 
\leq_{\text{lex}} 
(\beta_1,\dots,\beta_{m+1}). 
\]
Moreover, equality holds for some $W$.
\end{conjecture}

If we consider the first nonzero index of $\omega_r(d,m)$, that is, $l=\min\{i: \beta_i\neq 0\}$, then Conjecture~\ref{Conj: Bogulavsky} says that $\dim(V(W))\leq m-l$. By Lemma~\ref{Lem: first non-zero index}, this $l$ is determined by the range of $r$
$$\tbinom{m+d}{d}-\tbinom{m+d+1-l}{d} <r \leq \tbinom{m+d}{d}-\tbinom{m+d-l}{d}.$$
Thus, this part of Conjecture~\ref{Conj: Bogulavsky} follows from the following result of \cite{singhal2025conjecturebeelendattaghorpade}.

\begin{proposition}\cite[Proposition 4]{singhal2025conjecturebeelendattaghorpade}\label{dimension}
Suppose that we have $1\leq l\leq m$ and 
\[
\tbinom{m+d}{d}-\tbinom{m+d+1-l}{d} 
<r 
\leq \tbinom{m+d}{d}.
\] 
Given $F_1, \dots, F_r \in S_d(m)$ that are linearly independent, we have
\[
\dim(V(F_1, \dots, F_r)) 
\leq m-l.
\]
\end{proposition}

Moreover, for the same choice of $l$, Conjecture~\ref{Conj: Bogulavsky} says that $\deg_{m-l}(V(W))\leq \beta_l$.
This follows from Corollary~\ref{Cor: l,c} and the following result of \cite{singhal2025conjecturebeelendattaghorpade}.

\begin{proposition}\cite[Proposition 5]{singhal2025conjecturebeelendattaghorpade}\label{degree}
Suppose that we have $1\leq l\leq m$, $1\leq c\leq d$ and
$$\tbinom{m+d}{d}-\tbinom{m+d+1-l}{d}+\tbinom{m+d-l-c}{d-c-1} <r \leq \tbinom{m+d}{d}.$$
Then given $F_1, \dots, F_r \in S_d(m)$ that are linearly independent, we have:
$$\deg_{m-l}(V(F_1, \dots, F_r)) \leq c.$$
\end{proposition}

Now, let us restrict ourselves to the case $m=2$ and suppose $\omega_r(d,2)=(\beta_1,\beta_2,\beta_3)$. If $\beta_1=0$, then the previous two propositions prove Conjecture~\ref{Conj: Bogulavsky} for this $r$.
On the other hand, if $\beta_1>0$, then they say that $\deg_1(V(W))\leq \beta_1$. It remains to be shown that if $\deg_1(V(W))= \beta_1$, then $\deg_0(V(W))\leq \beta_2$. We prove this in the next lemma.

\begin{lemma}
Suppose $\omega_r(d,2)=(\beta_1,\beta_2,\beta_3)$ with $\beta_1>0$.
For subspaces $W$ of $S_d(2,\fq)$ of dimension $r$, if $\deg_1(V(W))= \beta_1$, then $\deg_0(V(W))\leq \beta_2$.
\end{lemma}
\begin{proof}
Consider a subspace $W$ of $S_d(2,\fq)$ of dimension $r$ with $\deg_1(V(W))= \beta_1$.
Denote $g=\gcd(W)$. Then we have $\deg(g)=\deg_1(V(W))=\beta_1$. Suppose $W=g W_1$, so $W_1$ consists of polynomials of degree $d-\beta_1$ and $\gcd(W_1)=1$. Note that
\[\deg_0(V(W)) =\deg(V(W_1)) =|V(W_1)|_{\mm}\leq u_{r} (d-\beta_1,2).\]
Therefore, Theorem~\ref{Thm: urd2} implies that $u_{r} (d-\beta_1,2)=H'_{r-1}(d-\beta_1,2)$. By Lemma~\ref{r as function of alpha}, we see that
\[r=1+\tbinom{1+d-\beta_1}{2} +\tbinom{d-\beta_1-\beta_2}{1}.\]
Consider the $s$ for which $\omega_s(d-\beta_1,2)=(0,\beta_2,\beta_3)$. By Lemma~\ref{r as function of alpha}, we see that
\[s=1+\tbinom{1+(d-\beta_1) -0}{2} +\tbinom{(d-\beta_1)-0-\beta_2}{1}.\]
This means that $s=r$. By Lemma~\ref{Lem: w and w'}, we see that $\omega'_{r-1}(d-\beta_1,2)=(0,\beta_2,\beta_3)$, that is, $H'_{r-1}(d-\beta_1,2)=\beta_2$. This completes the proof.
\end{proof}

The discussion in this section proves the following.

\begin{theorem}
Conjecture~\ref{Conj: Bogulavsky} is true when $m=2$.
\end{theorem}

\section*{Acknowledgments}
We thank Sudhir Ghorpade for introducing us to this problem. We thank Nathan Kaplan for many helpful discussions about the problem.

\appendix

\section{Technical Lemmas}\label{Appendix}

In this appendix, our goal is to prove Proposition~\ref{Prop: H' smaller than H}.

\begin{lemma}\label{Lem: diff ws}
Suppose $\omega_{s}(d,m)=(\beta_1,\dots,\beta_{m+1})$ and
\[k=\max\{1\leq k\leq m: \beta_k\neq 0\}.\]
Denote $\sigma=\sum_{i=1}^m\beta_i =d-\beta_{m+1}$.
Then for $1\leq j\leq m+1-k$, we have
\[H_s(d,m;d+1)-H_{s+j}(d,m;d+1)=(\sigma+1) (d+1)^{m-k-1} -\lfloor(d+1)^{m-k-j}\rfloor.\]
\end{lemma}
\begin{proof}
Note that $\omega_{s+j}(d,m)=(\beta_1,\dots,\beta_{k-1},\beta_k-1,d-\sigma,0,\dots,0,1,0,\dots,0)$.
Thus,
\[H_s(d,m;d+1)=\sum_{i=1}^{k}\beta_i (d+1)^{m-i},\]
\[H_{s+j}(d,m;d+1)=\sum_{i=1}^{k}\beta_i (d+1)^{m-i} -(d+1)^{m-k} +(d-\sigma)(d+1)^{m-k-1}+\lfloor(d+1)^{m-k-j}\rfloor.\]
The result follows.
\end{proof}

\begin{lemma}\label{lem:}
Given $a,b\geq 1$, we have
\[(d+1)^{a-1}+(d+1)^{b-1}\leq (d+1)^{a+b-1}.\]
\end{lemma}
\begin{proof}
Assume $a\leq b$. Then we have
\[(d+1)^{a-1}+(d+1)^{b-1}\leq 2 (d+1)^{b-1} \leq (d+1)^b \leq (d+1)^{a+b-1}.\qedhere\]
\end{proof}

\begin{lemma}\label{lem2:}
Given $1\leq c\leq a,b$, we have
\[(d+1)^{a-1}+(d+1)^{b-1}-(d+1)^{c-1}\leq (d+1)^{a+b-c-1}.\]
\end{lemma}
\begin{proof}
Assume $a\leq b$. If $c=a$, then equality holds so assume $c\leq a-1$.
Then we have
\[(d+1)^{a-1}+(d+1)^{b-1}\leq 2 (d+1)^{b-1} \leq (d+1)^b \leq (d+1)^{a+b-c-1}.\qedhere\]
\end{proof}

\begin{lemma}
Suppose $\omega_s(d,m)=(\beta_1,\dots,\beta_{m+1})$ and $l$ is the smallest index for which $\beta_l\neq 0$. Assuming $l<m$, we have
\[H_s(d,m;d+1)-H_{s+(m-l)}(d,m;d+1)\leq (\beta_{l}+1) (d+1)^{m-l-1}-1.\]
\end{lemma}
\begin{proof}
We will break the difference $H_s(d,m;d+1)-H_{s+(m-l)}(d,m;d+1)$ into parts $\sum_i H_{s_{i-1}}(d,m;d+1)-H_{s_i}(d,m;d+1)$, such that each part $H_{s_{i-1}}(d,m;d+1)-H_{s_i}(d,m;d+1)$ can be obtained from Lemma~\ref{Lem: diff ws}.
For this we will construct a sequence $s=s_0<s_1<\dots<s_t=s+(m-l)$. We will denote
\begin{align*}
\delta_i&= s_i-s_{i-1},\\
\Delta_i&= H_{s_{i-1}}(d,m;d+1)-H_{s_i}(d,m;d+1),\\
\omega_{s_{i-1}}(d,m)&= (\gamma_{i-1,1},\dots,\gamma_{i-1,m+1}),\\
k_i &= \max\{1\leq k\leq m\mid \gamma_{i-1,k}\neq 0\},\\
\sigma_i &= \sum_{j=1}^{m} \gamma_{i-1,j}= d-\gamma_{i-1,m+1}.
\end{align*}
In order to apply Lemma~\ref{Lem: diff ws}, we need $\delta_i\leq m+1-k_i$. Denote
\[u_i=m+1-k_i.\]

We start with $s_0=s$. Suppose $s_0,\dots,s_{i-1}$ have been constructed. Using these we can compute $k_i$, $\sigma_i$ and $u_i$. Take
\[s_i=s_{i-1}+\min\big(u_i, s_0+(m-l)-s_{i-1}\big).\]
Thus, $\delta_i\leq u_i$. This will terminate at some $t$ with $s_t=s_0+(m-l)$.

Now, Lemma~\ref{Lem: diff ws} says that for $1\leq i\leq t$ we have
\begin{align*}
\Delta_i= \begin{cases}
    (\sigma_i+1)(d+1)^{u_i-2} -(d+1)^{u_i-1-\delta_i} &\text{ if } \delta_i<u_i\\
    (\sigma_i+1)(d+1)^{u_i-2} &\text{ if } \delta_i=u_i.
\end{cases}
\end{align*}
Note that for $1\leq i\leq t-1$, we have $\delta_i=u_i$ and hence
\[\Delta_i= (\sigma_i+1)(d+1)^{u_i-2} \leq (d+1)\times (d+1)^{u_i-2} =(d+1)^{\delta_i-1}.\]
Notice that $\sum_{i=1}^{t-1} \delta_i= (m-l)-\delta_t$.
Therefore, by Lemma~\ref{lem:} we have
\[\sum_{i=1}^{t-1} \Delta_i
\leq \sum_{i=1}^{t-1} (d+1)^{\delta_i-1}
\leq (d+1)^{m-l-\delta_t-1}.\]
Note that if $t=1$, then the sum on the left is $0$, so the inequality still holds.

For $i=t$, there are three possibilities:
\begin{enumerate}
\item Case 1: $\delta_t=u_t$. In this case $\Delta_t\leq (d+1)^{\delta_t-1}$. Thus, by Lemma~\ref{lem:} we have
\begin{align*}
&H_{s}(d,m;d+1)-H_{s+(m-l)}(d,m;d+1)\\
&=\sum_{i=1}^t \Delta_i
\leq (d+1)^{(m-l-\delta_t)-1} +(d+1)^{\delta_t-1}\\
&\leq (d+1)^{m-l-1}
\leq (\beta_{l}+1) (d+1)^{m-l-1}-1.
\end{align*}

\item Case 2: $\delta_t<u_t$ and $k_{t}>l$.
In this case
\[\Delta_t= (\sigma_t+1)(d+1)^{u_t-2} -(d+1)^{u_t-1-\delta_t} \leq (d+1)^{u_t-1} -(d+1)^{u_t-1-\delta_t}.\]
Since $k_t>l$, we have $u_t=m+1-k_t\leq m-l$.
Therefore, by Lemma~\ref{lem2:} we have
\begin{align*}
&H_{s}(d,m;d+1)-H_{s+(m-l)}(d,m;d+1)
=\sum_{i=1}^t \Delta_i\\
&\leq  (d+1)^{m-l-\delta_t-1} +(d+1)^{u_t-1} -(d+1)^{u_t-\delta_t-1}\\
&\leq (d+1)^{m-l-1}
\leq (\beta_{l}+1) (d+1)^{m-l-1}-1.
\end{align*}

\item Case 3: $\delta_t<u_t$ and $k_{t}=l$.
In this case
\[\Delta_t=(\sigma_t+1)(d+1)^{u_t-2} -(d+1)^{u_t-1-\delta_t} = (\beta_l+1) (d+1)^{u_t-2} -(d+1)^{u_t-1-\delta_t}.\]
Notice that $k_t=l$ means $u_t=m+1-k_t=m-l+1$.
Now, if $t=1$, then $\delta_t=m-l$, hence
\[H_{s}(d,m;d+1)-H_{s+(m-l)}(d,m;d+1)
=\Delta_t = (\beta_l+1) (d+1)^{m-l-1} -1.\]
If $t\geq 2$, then $\delta_t\leq (m-l)-\delta_{t-1}\leq m-l-1$, so
\begin{align*}
&H_{s}(d,m;d+1)-H_{s+(m-l)}(d,m;d+1)
=\sum_{i=1}^t \Delta_i\\
&\leq (d+1)^{m-l-\delta_t-1} +(\beta_l+1)(d+1)^{m-l-1} -(d+1)^{m-l-\delta_t}\\
&=(\beta_l+1)(d+1)^{m-l-1} -d(d+1)^{m-l-1-\delta_t}\\
&\leq (\beta_l+1)(d+1)^{m-l-1} -1.\qedhere
\end{align*} 
\end{enumerate}
\end{proof}

\begin{lemma}\label{lem:3}
Given $d,k\geq 1$, we have
\[d^k\leq 1+(d-1)(d+1)^{k-1}.\]
\end{lemma}
\begin{proof}
The case $d=1$ and the cases $k\in \{1,2,3\}$ are easy to verify directly. Now assume $d\geq 2$ and $k\geq 4$. It is enough to show that
\[1\leq \big(1-\tfrac{1}{d}\big)\big(1+\tfrac{1}{d}\big)^{k-1}.\]
Consider the function $f(x)=(1-x)(1+x)^{k-1}$, it is easy to see that $f'(x)\geq 0$ for $0\leq x\leq \frac{k-2}{k}$. Since $d\geq 2$ and $k\geq 4$, we know that $\frac{1}{d}\leq \frac{1}{2}\leq \frac{k-2}{k}$. Thus, $f(0)\leq f(\frac{1}{d})$. The result follows.
\end{proof}

\begin{corollary}
Given $\beta,k,d\geq 1$, we have
\[(\beta+1) (d+1)^{k-1} -1\leq \beta \big( (d+1)^k -d^k\big).\]
\end{corollary}
\begin{proof}
Lemma~\ref{lem:3} can be restated as
\[(d+1)^{k-1} -1\leq d(d+1)^{k-1}-d^k.\]
Therefore,
\[(d+1)^{k-1} -1 \leq d(d+1)^{k-1}-d^k \leq \beta (d(d+1)^{k-1}-d^k)
=\beta ((d+1)^{k}-d^k-(d+1)^{k-1}).\]
The result follows.
\end{proof}

\begin{customprop}{\ref{Prop: H' smaller than H}}
Given $m\leq r\leq \binom{m+d}{d}$, for every $q\geq d+1$, we have
\[H'_{r-(m-1)}(d,m)\leq H_r(d,m;q).\]
\end{customprop}
\begin{proof}
Suppose $\omega'_{r-(m-1)}(d,m)=(\beta_1,\dots,\beta_{m+1})=\omega_s(d,m)$. Let $l$ be the smallest index for which $\beta_l\neq 0$. Then we know that $r=s+(m-l)$.

If $l=m$, then $r=s$ and hence
\[H'_{r-(m-1)}(d,m)=\beta_{m}= H_r(d,m;q).\]

Next, if $l<m$, then we have
\begin{align*}
&H_s(d,m;d+1)- H_{s+(m-l)}(d,m;d+1)\\
&\leq (\beta_l+1) (d+1)^{m-l-1} -1\\
&\leq \beta_l ((d+1)^{m-l}- d^{m-l})\\
&\leq H_s(d,m;d+1) -H_s(d,m;d).
\end{align*}
This means $H_s(d,m;d)\leq H_{s+(m-l)}(d,m;d+1)$. We conclude that
\[H'_{r-(m-1)}(d,m) =H_s(d,m;d)\leq H_{s+(m-l)}(d,m;d+1)= H_r(d,m;d+1)\leq H_r(d,m;q).\]
\end{proof}

\bibliographystyle{plain}
\bibliography{Bibl.bib}

\end{document}